\documentclass[12pt]{amsart}
\usepackage{amssymb}
\usepackage[all]{xy}

\input xy
\xyoption{all}

\newtheorem{lemma}{Lemma}[section]
\newtheorem{corollary}[lemma]{Corollary}
\newtheorem{theorem}[lemma]{Theorem}
\newtheorem{proposition}[lemma]{Proposition}

\theoremstyle{definition}
\newtheorem{remark}[lemma]{Remark}

\newtheorem{example}[lemma]{Example}
\newtheorem*{C1}{Conjecture 1 (C1)}
\newtheorem*{C2}{Conjecture 2 (C2)}
\newtheorem*{IC}{Isomorphism Conjecture for Graph Algebras (IC)}
\newtheorem*{GIC}{Generalized Isomorphism Conjecture (GIC)}
\newtheorem*{GSIC}{Generalized Strong Isomorphism Conjecture (GSIC)}

\newcommand\Cset{\mathbb {C}}
\newcommand\Zset{\mathbb {Z}}
\newcommand\Qrtot{Q^r_{\mathrm{tot}}}
\newcommand\Qltot{Q^l_{\mathrm{tot}}}
\newcommand\Qtot{Q^{\sigma}_{\mathrm{tot}}}
\newcommand\Qrmax{Q^r_{\mathrm{max}}}
\newcommand\Qlmax{Q^l_{\mathrm{max}}}
\newcommand\Qmax{Q^{\sigma}_{\mathrm{max}}}
\newcommand\Qrcl{Q^r_{\mathrm{cl}}}
\newcommand\Qlcl{Q^l_{\mathrm{cl}}}
\newcommand\ef{{\mathfrak F}}
\newcommand\dirlim{\mathop{\varinjlim}\limits}
\newcommand\f{{\mathcal F}}
\newcommand\te{{\mathcal T}}
\newcommand\homo{\mathrm{Hom}}

\begin{document}
\title{Noetherian Leavitt path algebras and their regular algebras}

\author{Gonzalo Aranda Pino}
\address{Departamento de \'{A}lgebra, Geometr\'{\i}a y Topolog\'{\i}a\\
Universidad de M\'{a}laga\\
29071, M\'{a}laga, Spain}
\email{g.aranda@uma.es}

\author{Lia Va\v s}
\address{Department of Mathematics, Physics and Statistics\\
University of the Sciences in Philadelphia\\
Philadelphia, PA 19104, USA}
\email{l.vas@usp.edu}

\thanks{The first author was partially supported by the Spanish MEC and Fondos FEDER through project MTM2007-60333, and by the Junta de Andaluc\'{\i}a and Fondos FEDER, jointly, through projects FQM-336 and FQM-2467. Part of this work was carried out during a visit of the second author to the University of M\'alaga, partially funded by a ``Grant for foreign visiting professors'' within the III Research Framework Program of the University of M\'alaga. The second author thanks this host institution for its warm hospitality and support.}



\begin{abstract} In the past, it has been shown that the Leavitt path algebra $L(E)=L_K(E)$ of a graph $E$ over a field $K$ is left and right noetherian if and only if the graph $E$ is finite and no cycle of $E$ has an exit. If $Q(E)=Q_K(E)$ denotes the regular algebra over $L(E),$ we prove that these two conditions are further equivalent with any of the following: $L(E)$ contains no infinite set of orthogonal idempotents, $L(E)$ has finite uniform dimension, $L(E)$ is directly finite, $Q(E)$ is directly finite, $Q(E)$ is unit-regular, $Q(E)$ is left (right) self-injective and a few more equivalences. In addition, if the involution on the field $K$ is positive definite, these conditions are equivalent with the following: the involution $^\ast$ extends from $L(E)$ to $Q(E),$ $Q(E)$ is $^\ast$-regular, $Q(E)$ is finite, $Q(E)$ is the maximal (total or classical) symmetric ring of quotients of $L(E),$ the maximal right ring of quotients of $L(E)$ is the same as the total (or classical) left ring of quotients of $L(E),$ every finitely generated nonsingular $L(E)$-module is projective, and the matrix ring $M_n(L(E))$ is strongly Baer for every $n$. It may not be surprising that a noetherian Leavitt path algebra has these properties, but a more interesting fact is that these properties hold {\em only} if a Leavitt path algebra is noetherian (i.e. $E$ is a finite no-exit graph).

Using some of these equivalences, we give a specific description of the inverse of the isomorphism $V(L(E))\rightarrow V(Q(E))$ of monoids of equivalence classes of finitely generated projective modules of $L(E)$ and $Q(E)$ for noetherian Leavitt path algebras. We also prove that two noetherian Leavitt path algebras are isomorphic as rings if and only if they are isomorphic as $^\ast$-algebras. This answers in affirmative the Isomorphism Conjecture for the class of noetherian Leavitt path algebras: if $L_{\mathbb C}(E)$ and $L_{\mathbb C}(F)$ are noetherian Leavitt path algebras, then $L_{\mathbb C}(E)\cong L_{\mathbb C}(F)$ as rings implies $C^*(E)\cong C^*(F)$ as $^\ast$-algebras.
\end{abstract}

\maketitle

\section{Introduction}

In the last decade, Leavitt path algebras have been generating significant interest. Introduced in \cite{Gene_Gonzalo1} and \cite{AMP}, these algebras represent an algebraic analog of a class of $C^*$-algebras. Leavitt path algebras are free algebras over a field that satisfy the same relations as the graph $C^*$-algebras (the introduction to M. Tomforde's paper \cite{Tomforde} is a good source for more details on graph $C^*$-algebras). In addition to being algebraic counterparts of graph $C^*$-algebras, Leavitt path algebras are generalization of the classical Leavitt algebras, the algebras that fail the invariant basis number property universally in a certain respect (for more details see \cite{Bergman}).

In \cite{AB}, a row-finite Leavitt path algebra $L(E)$ is embedded in a (von Neumann) regular algebra $Q(E)$ with isomorphic monoids of isomorphism classes of finitely generated projective modules. The algebra $Q(E)$ is called the regular algebra of a Leavitt path algebra. In this paper, we describe exactly when $Q(E)$ is unit-regular, self-injective, directly finite, and equal to the maximal (left and right) ring of quotients of $L(E)$. We relate these conditions with those stating that a Leavitt path algebra is right and left noetherian, directly finite and without infinite set of orthogonal idempotents.

By \cite[Theorems 3.8 and 3.10]{AAS2}), it is known that the Leavitt path algebra $L(E)$ over a graph $E$ is (left and right) noetherian exactly when $E$ is finite and no cycle in $E$ has an exit. This last condition is known as (NE) in the literature. In this paper the graphs satisfying Condition (NE) are called no-exit graphs for short. We obtain further characterizations of noetherian Leavitt path algebras as well as the characterizations of the above mentioned algebraic properties of $Q(E)$ over a finite no-exit graph $E$. In particular, we obtain a set of nine new equivalent conditions relating many algebraic properties of $L(E)$ and $Q(E)$ to the no-exit condition on a finite graph $E$ (Theorem \ref{equivalences1-14}). Moreover, if the involution on the base field $K$ is positive definite, we also prove that the involution of $L(E)$ extends to $Q(E)$ making $Q(E)$ $^\ast$-regular, symmetric and finite just in case that $E$ is a finite no-exit graph (Theorem \ref{equivalences15-19}).

After that, we consider when $Q(E)$ is equal to the maximal, total and classical left, right and symmetric ring of quotients of $L(E)$ adding another ten equivalent conditions (Corollary \ref{Q(E)_as_quotient} and Proposition \ref{extendible_LPAs}) to the list of the equivalences.
A Leavitt path algebra of a finite no-exit graph is hereditary and noetherian so it may not be surprising that it has all the properties mentioned above. However, we emphasize the fact that from our results it follows that a Leavitt path algebra enjoys all these properties {\em only} if the underlying graph is finite and without exits.

Using some of these results, we obtain an explicit description of the inverse of the isomorphism $V(L(E))\rightarrow V(Q(E))$ of monoids of equivalence classes of finitely generated projective modules over a Leavitt path algebra and its regular algebra (Theorem \ref{K0_theorem}).

Finally, we give some generalizations to the isomorphism conjectures posed in \cite[Conjecture 1]{Abrams_Tomforde}, and show that, for the class of Leavitt path algebras considered in this paper, a strongly generalized version of the Isomorphism Conjecture of Graph Algebras \cite[p. 22]{Abrams_Tomforde} holds. Concretely, for the family of noetherian Leavitt path algebras, if $L(E)$ and $L(F)$ are isomorphic as rings, then they are isomorphic as $^\ast$-algebras as well. This, in turn, expands the family of graphs for which a positive answer has been given to the Isomorphism Conjecture for Graph Algebras. Specifically, if $E$ and $F$ are finite no-exit graphs, then $L_{\mathbb C}(E)\cong L_{\mathbb C}(F)$ as rings implies that $C^*(E)\cong C^*(F)$ as $^\ast$-algebras.

The paper is organized as follows. In Section \ref{section_recap}, we recall the main definitions, the construction of the regular algebra and some existing results. In Section \ref{section_charac_noeth}, we prove the equivalence of fourteen  conditions describing algebraic properties of $L(E)$ for a finite no-exit graph $E$ and of its regular algebra $Q(E)$. In Section \ref{section_symmetric}, we describe $Q(E)$ as a ring of quotients of $L(E)$ and add additional three equivalences to previously obtained fourteen. In Section  \ref{section_K_0}, we add seven final conditions to the set of equivalences and use them to describe the inverse of the map $V(L(E))\rightarrow V(Q(E))$ of monoids of equivalence classes of finitely generated projective modules. Lastly, in Section \ref{section_isomorphism_conjecture} we answer in the affirmative both the Isomorphism Conjecture for Graph Algebras and the Strongly Generalized Isomorphism Conjecture for the class of noetherian Leavitt path algebras.

\section{Summary of the construction of the regular algebra and some related results}
\label{section_recap}

Throughout this paper, $K$ denotes a field and $E$ a directed graph. We write $E=(E^0, E^1, r, s)$ to denote that $E^0$ is the set of vertices, $E^1$ the set of edges, and $r$ and $s$ maps $E^1\rightarrow E^0$ describing ranges and sources of edges respectively. We say that $E$ is finite if both $E^0$ and $E^1$ are finite and that $E$ is row-finite if every vertex emits only finitely many edges ($|s^{-1}(v)|<\infty$ for all $v\in E^0$). A vertex $v$ is called a sink if $s^{-1}(v)=\emptyset.$ It is called a source if $r^{-1}(v)=\emptyset.$

A path of length $n$ is a sequence of edges of the form $p=e_1\ldots e_n$ for some positive integer $n$ such that $s(e_{i+1})=r(e_i)$ for $i=1,\ldots, n-1.$ In this case we define the source $s(p)$ of $p$ to be $s(e_1),$ and the range $r(p)$ of $p$ to be $r(e_n).$ If $s(p)=r(p),$ $p$ is said to be closed. If $p$ is closed and $s(e_i)\neq s(e_j)$ for $i\neq j,$ then $p$ is called a cycle. An edge $e$ is an exit of a path $p=e_1\ldots e_n$ if there exists $i$ such that $s(e) = s(e_i)$ and $e\neq e_i.$

Considering vertices to be the paths of length 0, let $E^*$ denote the set of paths of all non-negative lengths. The path algebra $P_K(E)$ is a free $K$-algebra over the set of all paths $E^*$ where the multiplication of paths $p$ and $q$ is the concatenation if $r(p)=s(q)$ and it is 0 otherwise. Alternatively, $P_K(E)$ is a free $K$-algebra with basis consisting of vertices and edges of $E$ such that
\begin{itemize}
\item[(P1)] $vv=v$ and $v w=0$ if $v\neq w,$
\item[(P2)]  $e=s(e)e=er(e),$
\end{itemize}
for all vertices $v$ and $w$ and all edges $e.$ The notation $P_K(E)$ is shortened to $P(E)$ when we work over the same (fixed) field.

For a given graph $E,$ consider an extended graph of $E$ to be the graph with the same vertices and with edges $\{e\ |\ e\in E^1\}\cup \{e^*\ |\ e\in E^1\}$ where the range and source relations are the same as in the original graph for $e\in E^1$, and $r(e^*)=s(e)$ and $s(e^*)=r(e)$ for the added edges. The edges $e^*$ are called ghost edges. The Leavitt path algebra $L_K(E)$ is the free $K$-algebra with basis consisting of vertices, edges and ghost edges that satisfies the path algebra axioms (P1) and (P2) in addition to
\begin{itemize}
\item[(CK1)] $e^*e=r(e),$ and $e^*f=0$ if $e\neq f$ for all $e,f\in E^1,$

\item[(CK2)] $v=\sum ee^*$ for all $e\in E^1$ with $v=s(e)$ and all $v\in E^0$ with $0<|s^{-1}(v)|<\infty.$
\end{itemize}
These last two axioms are called  Cuntz-Krieger relations. The notation $L_K(E)$ is also often shortened to $L(E)$ when there is no danger of confusion.

It a well-known fact that $L(E)$ is unital with the identity element $1=\sum_{v\in E^0} v$ if and only if $E^0$ is finite (e.g. see \cite[Lemma 1.6]{Gene_Gonzalo1}) and that $L(E)$ is a $\mathbb{Z}$-graded $K$-algebra, spanned as a $K$-vector space by $\{pq^{\ast } \ \vert \ p,q \text{ are paths in }E\}$. (Recall that the elements of $E^{0} $ are viewed as paths of length $0$, so that this set includes elements of the form $v$ with $v\in E^{0}$.) In particular, for each $n\in\mathbb{Z}$, the degree $n$ component $L_{K}(E)_{n}$ is spanned by $\{pq^{\ast }\ \vert \ p,q \text{ are paths in }E\text{ with } l(p)-l(q)=n\}$, where $l(p)$ denotes the length of $p$.

Note that $L(E)$ is a ring with involution. Namely, for an involution $k\mapsto \overline{k}$ of the field $K$ (which may be taken to be the identity), one can define $(\sum k_{p,q} pq^*)^*=\sum \overline{k_{p,q}}qp^*.$ In \cite[Lemma 1.3.1]{Work}, I. Raeburn notes that, for the class of row-finite graphs without sources, if the involution on $K$ is {\em positive definite} (i.e. $\sum_{i=1}^n \overline{k_i}k_i=0$ implies $k_i=0$ for all $i=1,\ldots,n$ and for all non-negative integers $n$), then the involution on $L(E)$ is {\em proper} (i.e. $x^{\ast}x=0$ implies $x=0$). In \cite[Proposition 2.3]{ARV} it is shown that this result holds for all graphs.

The algebra $L(E)$ is always nonsingular as a ring (\cite[Proposition 4.1]{Mercedes_quotients}) and it is hereditary if $E$ is finite (\cite[Theorem 3.5]{AMP}). In the last decade, numerous ring theoretic properties of $L(E)$ (for example being finite-dimensional, simple, semisimple, purely infinite simple, regular, noetherian, artinian, exchange, prime, primitive, to name just a few) have been characterized in terms of graph theoretic properties of the underlying graph $E.$ We recall one such result.

Recall that a graph $E$ is said to satisfy Condition (NE) (NE for ``no-exit''), or that $E$ is a \emph{no-exit graph} for short, if no cycle in $E$ has an exit.
Also, recall that a ${\mathbb Z}$-graded $K$-algebra $A=\bigoplus_{n\in \mathbb{Z}} A_n$ is \emph{locally finite} in case $\dim_K(A_n)<\infty$ for every $n\in {\mathbb Z}$.

\begin{theorem} {\bf (\cite[Theorems 3.8 and 3.10]{AAS2})} For a finite graph $E$ and field $K$ the following conditions are equivalent:
\begin{enumerate}
\item $E$ is a no-exit graph.
\item $L(E)$ is left noetherian.
\item $L(E)$ is right noetherian.
\item $L(E)$ is locally finite.
\item If $l$ is the number of cycles in $E$ (call
them $c_1,\dots,c_l$), $m_i$ the number of paths ending in a fixed (although arbitrary) vertex $v_{m_i}$ of
the cycle $c_i$ which do not contain the cycle itself (for $1\leq i \leq l$), $k$ is the number of
sinks in $E$ (call them $w_{l+1}, \dots, w_{l+k}$), and for every $j\in \{1, \dots, k\}$, $n_j$
is the number of paths ending in the sink $w_{l+j}$, then $L(E)$ is isomorphic to a direct sum of the matrix rings below
$$\left(\bigoplus_{i=1}^l M_{m_i}(K[x, x^{-1}])\right)\oplus
\left(\bigoplus_{j=1}^{k}M_{n_j}(K)\right).$$
\end{enumerate}
\label{noetherian_NE_theorem}
\end{theorem}
This result was generalized for locally noetherian Leavitt path algebras over row-finite graphs in \cite[Theorem 3.7]{NE-paper}.

Lastly, we recall the concept of the regular ring of a (Leavitt) path algebra. In \cite{AB}, the regular ring $Q_K(E)$ of a path algebra $P(E)$ (and its Leavitt path algebra $L(E)$) is constructed. We outline the main idea of the construction from \cite{AB}. In the following, we fix the field $K$ and shorten the notation $Q_K(E)$ to $Q(E).$

\begin{enumerate}
\item[(i)] Let $E$ be a finite graph. Let $\Sigma$ be the set of matrices with entries in the path algebra $P(E)$ that become invertible in the algebra of power series $P((E))=\{\, \sum k_pp\, |\, p\in E^*$ possibly infinitely many $k_p$ are nonzero $\}.$

The universal localization of $P(E)$ with respect to $\Sigma$ (obtained by adding the entries of the inverse matrices of matrices in $\Sigma$ to $P(E)$) is the division and the rational closure of $P(E)$ in $P((E))$ (see \cite[Observation 1.18 and Theorem 1.20]{AB}). In \cite{AB}, it is denoted by $P_{rat}(E).$ By \cite[Proposition 2.15]{AB}, $P_{rat}(E)$ is always semihereditary.

\item[(ii)] Let $\Sigma_1$ be the set of the following homomorphisms of finitely generated projective left $P(E)$-modules. For every non-sink vertex $v,$ let $e_1,\ldots e_n$ be all the edges that $v$ emits. The homomorphism $\mu$ mapping $P(E)v$ to $P(E)r(e_1)\oplus\ldots\oplus P(E)r(e_n)$ by $r\mapsto (re_1,\ldots, re_n)$ is in $\Sigma_1$ and the algebra $L(E)$ is the universal localization of $P(E)$ with respect to $\Sigma_1.$

Also, let $\overline{E}$ be the opposite graph of $E,$ i.e. the graph with $\overline{E}^0=E^0,$ $\overline{E}^1=\{\overline{e}\ |\ e\in E^1\},$ $\overline{s}(\overline{e})=r(e)$ and  $\overline{r}(\overline{e})=s(e).$ Let $\Sigma_2$ denote the set of the following homomorphisms of finitely generated projective left $P(\overline{E})$-modules: for every non-sink vertex $v,$ let $e_1,\ldots e_n$ be all the edges that $v$ emits and $\nu$ be the mapping $P(\overline{E})r(e_1)\oplus\ldots\oplus P(\overline{E})r(e_n)$ to $P(\overline{E})v$ by $(r_1,\ldots, r_n)\mapsto \sum r_i \overline{e}_i.$ Then $L(E)$ is the universal localization of $P(\overline{E})$ with respect to $\Sigma_2.$

\item[(iii)] The ring $Q(E)$ is the universal localization of $P_{rat}(E)$ with respect to $\Sigma_1.$ It is also the universal localization of $L(E)$ with respect to $\Sigma.$ Finally, it is also universal localization of $P(E)$ with respect to $\Sigma\cup\Sigma_1.$

Moreover, if $E$ has $d$ vertices, the following diagram commutes.
$$\xymatrix{
K^d\ar[d]\ar[r]&P(E)\ar[r]^{\Sigma^{-1}}\ar[d]^{\Sigma_1^{-1}}&P_{rat}(E)\ar[r]\ar[d]^{\Sigma_1^{-1}}&P((E))\ar[d]^{\Sigma_1^{-1}}\\
P(\overline{E})\ar[r]^{\Sigma_2^{-1}} & L(E)\ar[r]^{\Sigma^{-1}}  &  Q(E)\ar[r] & U(E) }$$
Here the algebra $U(E)$ is the universal localization of $P((E))$ with respect to $\Sigma_1^{-1}.$

\item[(iv)] The ring $Q(E)$ satisfies the path algebra axioms (P1), (P2), together with (CK1) and (CK2).

\item[(v)] The ring $Q(E)$ is regular, hereditary and such that the monoids of finitely generated projectives $V(L(E))$ and $V(Q(E))$ are isomorphic (see \cite[Theorem 3.5]{AMP} and \cite[Theorem 4.2]{AB}).

\item[(vi)] The algebra $L(E)$ is a perfect right ring of quotients of $P(E)$ and a perfect left ring of quotients of $P(\overline{E})$. The ring $Q(E)$ is the total left ring of quotients of $P(\overline{E})$ and total left ring of quotients of $L(E)$ (all the necessary background on rings of quotients can be found in \cite{Stenstrom}). So we have that
\[Q(E)=\Qltot(L(E))=\Qltot(P(\overline{E}))\]
The proofs of these claims can be found in \cite{Miquel}.

\item[(vii)] If $E$ is row-finite (but not finite necessarily), $Q(E)$ is the direct limit of the regular rings of path algebras of finite subgraphs of $E$ (\cite[paragraph before Theorem 4.4]{AB}).
\end{enumerate}

\section{Characterizations of noetherian Leavitt path algebras}
\label{section_charac_noeth}

In this section, $E$ denotes a finite graph. First, we prove two sufficient conditions for $E$ to be no-exit. Then, we prove some necessary conditions for $E$ to be no-exit. As a corollary, we obtain a list of equivalences to the condition that $E$ is no-exit. After that, we study some implications of the statement that the involution extends from a Leavitt path algebra to its regular algebra. As a corollary, we obtain further characterizations of noetherian Leavitt path algebras in case when the involution on underlying field is positive definite.

\begin{proposition} Any of the following two conditions imply that $E$ is a no-exit graph.
\begin{itemize}
\item[(i)] $L(E)$ is finite.

\item[(ii)] $L(E)$ contains no infinite set of orthogonal idempotents.
\end{itemize}
\label{sufficient}
\end{proposition}
\begin{proof} (i) Let us assume that $L(E)$ is finite but that $E$ has a cycle $p$ with an exit $e$. By rotating the cycle if necessary, we can assume that the exit $e$ occurs at the base of the cycle $p$, so that $p=e_1e_2\dots e_n$ and $v=s(e_1)=s(e)$ with $e_1\neq e.$ Consider the element $x=p+\sum_{w\neq v}w$ that satisfies: $x^*x=(p^*+\sum_{w\neq v}w)(p+\sum_{w\neq
v}w)=p^*p+\sum_{w\neq v}w=v+\sum_{w\neq v}w=1.$ Then $xx^*=(p+\sum_{w\neq v}w)(p^*+\sum_{w\neq
v}w)=pp^*+\sum_{w\neq v}w=pp^*+1-v=1$ and so $pp^*=v.$ Multiplying by $e^*$ on the left we
have $0=(e^*e_1)(e_2\dots e_n)p^*=e^*pp^*=e^*v=e^*\neq 0$, a contradiction. Thus, there cannot be a cycle with an exit and so $E$ is a no-exit graph.

(ii) Suppose, by way of contradiction, that $E$ is not a no-exit graph so that there exists a
cycle $p$ with an exit $e$. By relabeling the vertices if necessary, we can assume that $s(e)$ is the base of
the cycle so that we can write $p=e_1\dots e_n$ with $s(p)=s(e_1)=s(e)$ and $e_1\neq e$. In this case we
clearly have $e_1^*e=0=e^*e_1$ which in turn implies $p^*e=e^*p=0$. Consider the set $\f=\{p^nee^*(p^*)^n \}_{n=1}^{\infty}.$ It is easy to check that the elements of $\f$ are orthogonal idempotents. However, $\f$ is infinite. To see that let $n>m$ and assume that $p^nee^*(p^*)^n=p^mee^*(p^*)^m.$ Multiplying by $e^*(p^*)^m$ on the left, we obtain
$$0=(e^*p^{n-m})ee^*(p^*)^n=e^*(p^*)^mp^nee^*(p^*)^n=e^*(p^*)^mp^mee^*(p^*)^m=e^*(p^*)^m,$$ a contradiction
with the fact that ghost paths are linearly independent in $L(E)$ by \cite[Lemma 1.6]{G2}. Thus, we obtain an infinite set of orthogonal idempotents, a contradiction with the hypothesis.
\end{proof}

\begin{proposition} The condition that $E$ is a no-exit graph implies any of the following two conditions.
\begin{itemize}
\item[(i)] $Q(E)$ is unit-regular, (left and right) self-injective and $Q(E)=\Qltot(L(E))=\Qlmax(L(E)).$

\item[(ii)] If $\phi$ is the isomorphism
$$\phi: L(E)\cong \left(\bigoplus_{i=1}^l M_{m_i}(K[x, x^{-1}])\right)\oplus
\left(\bigoplus_{j=1}^{k}M_{n_j}(K)\right)$$ as in Theorem \ref{noetherian_NE_theorem}
where $l$ is the number of cycles  $c_1,\dots,c_l$  in $E$, $m_i$ the number of paths ending in a fixed vertex $v_{m_i}$ of
the cycle $c_i$ which do not contain the cycle itself for $1\leq i \leq l$, $k$ is the number of
sinks $w_{l+1}, \dots, w_{l+k}$ in $E$, and for every $j\in \{1, \dots, k\}$, $n_j$
is the number of paths ending in the sink $w_{l+j}$, then this algebra isomorphism is  a $^\ast$-isomorphism.
\end{itemize}
\label{necessary}
\end{proposition}
\begin{proof} (i) Assume that $E$ is a no-exit graph. Then $L(E)$ is noetherian by Theorem \ref{noetherian_NE_theorem} and hereditary by \cite[Theorem 3.5]{AMP}. In this case, a result from \cite[Example 3, p. 235]{Stenstrom} states that the maximal and total left rings of quotients are equal. So, $Q(E)=\Qltot(L(E))=\Qlmax(L(E)).$ Thus $Q(E)$ is left self-injective since $\Qlmax(L(E))$ is. On the other hand, a left self-injective and left hereditary ring is semisimple (\cite[Theorem 7.52]{Lam}) and thus it is right self-injective as well. But a regular left and right self-injective ring is unit-regular by \cite[Theorem 9.29]{G}.

(ii) Following the proof of \cite[Theorem 3.8]{AAS2}, the basis of $L(E)$ can be described as follows. Let $\Lambda_i,$ $i=1,\ldots,l,$ be the set of paths ending in a fixed vertex $v_{m_i}$ of the cycle $c_i$ which do not contain the cycle itself, and let $\Lambda_j,$ $j=l+1,\ldots, l+k,$ be the set of paths ending in a sink $w_{l+j}.$ Let $\Lambda=\bigcup_{t=1}^{l+k} \Lambda_t.$ If $\lambda$ denotes the cardinality of $\Lambda,$ index the elements of $\Lambda$ as $p_s,$ $s=1,\ldots, \lambda.$ Finally, let
\[X = \{\, p_r c^z_t  p_s^*\, |\, z \in \Zset,\, r,s = 1,\ldots, \lambda,\,  t = 1,\ldots, l+k\,\}\]
where $c_t$ denotes $w_t$ for $t > l$ and $c_t^{-z}$ denotes $(c_t^*)^z$ for positive $z$ and
$t=1,\ldots, l.$
In \cite[Theorem 3.8]{AAS2} it is shown that the nonzero elements of $X$ constitute a basis of $L(E)$ and that $\phi$ is the isomorphism mapping a nonzero element $p_r c^z_t  p_s^*$ of $X$ for $t\leq l$ to $x^z e_{rs}$ where $e_{rs}$ is the standard matrix unit in appropriate matrix algebra $M_{m_i}(K[x,x^{-1}])$ for appropriate $m_i.$ For $t>l,$ $\phi$ maps a nonzero element  $p_r w_t p_s^*=p_rp_s^*$ of $X$ to the standard matrix unit $e_{rs}$ in the appropriate matrix algebra $M_{n_j}(K)$ for appropriate $n_j.$

The isomorphism $\phi$ is obtained as an $K$-algebra extension of this map of the basis elements. This map has to be $^\ast$-isomorphism then also because on the basis elements,
$\phi((p_r c^z_t  p_s^*)^*)=\phi(p_s c_t^{-z}p_r^*)=x^{-z}e_{sr}=(x^ze_{rs})^*=(\phi(p_r c^z_t  p_s^*))^*$ if $t\leq l.$ If $t>l$ the claim follows similarly.
\end{proof}

This proposition gives us that the involution of a noetherian Leavitt path algebra $L(E)$ corresponds to the conjugate transpose involution of the sum of matrix algebras $\left(\bigoplus_{i=1}^l M_{m_i}(K[x, x^{-1}])\right)\oplus
\left(\bigoplus_{j=1}^{k}M_{n_j}(K)\right)$ where the involution on $K[x,x^{-1}]$ is given by $q\mapsto \overline{q}$ for $q\in K$ and $x\mapsto x^{-1}.$

\begin{theorem} Let $E$ be a finite graph.
The following conditions are equivalent.
\begin{enumerate}
\item[(1)] $E$ is a no-exit graph.
\item[(6)] $L(E)$ contains no infinite set of orthogonal idempotents.
\item[(7)] $Q(E)$ is unit-regular.
\item[(8)] $Q(E)$ is directly finite (i.e. $xy=1$ implies that $yx=1$ for all $x$ and $y$).
\item[(9)] $L(E)$ is directly finite.
\item[(10)] $L(E)$ is finite (i.e. $x^*x=1$ implies that $xx^*=1$ for all $x$).
\item[(11)] $Q(E)$ is left and right self-injective.
\item[(12)] $Q(E)$ is semisimple.
\item[(13)] $L(E)$ has finite uniform dimension (as a left and as a right $L(E)$-module).
\item[(14)] The monoid of equivalence classes of finitely generated projectives $V(L(E))\cong V(Q(E))$ is cancellative.
\end{enumerate}
\label{equivalences1-14}
\end{theorem}

The gap in the numbering of the conditions (the first condition in Theorem \ref{equivalences1-14} is labeled by (1) and the second by (6)) indicates conditions (2)--(5) from Theorem \ref{noetherian_NE_theorem}. Any mention of (2)--(5) in the proof of Theorem \ref{equivalences1-14} refers to the conditions from Theorem \ref{noetherian_NE_theorem}.

\begin{proof}
(1) $\Rightarrow$ (6). Condition (1) implies that $L(E)$ is left and right noetherian by Theorem \ref{noetherian_NE_theorem}. But a noetherian ring contains no infinite set of orthogonal idempotents.

The implication (6) $\Rightarrow$ (1) is Proposition \ref{sufficient}.

The implication (1) $\Rightarrow$ (7) is Proposition \ref{necessary}.

The implication (7) $\Rightarrow$ (8) is \cite[Proposition 5.2]{G}.

The implications (8) $\Rightarrow$ (9) $\Rightarrow$ (10) hold by definition. 

The implication (10) $\Rightarrow$ (1) is Proposition \ref{sufficient}.
Thus, conditions (1) -- (10) are equivalent.

The implication (1) $\Rightarrow$ (11) is Proposition \ref{necessary}. 

The implication (11) $\Rightarrow$ (7) is \cite[Theorem 9.29]{G}. 

(11) $\Leftrightarrow$ (12). One direction holds because a hereditary and self-injective ring is semisimple. Conversely, a semisimple ring is both hereditary and self-injective.

The equivalences (13) $\Leftrightarrow$ (2) and (3) follow from \cite[Theorem
7.58]{Lam}.

(14) $\Leftrightarrow$ (7). \cite[Theorem 2]{H} asserts that $Q(E)$ is unit-regular if and only if $V(Q(E))$ is cancellative. The isomorphism of monoids $V(Q(E))\cong V(L(E))$ is shown in \cite[Theorem 4.2]{AB} and \cite[Theorem 3.5]{AMP}.

Thus, conditions (1) -- (14) are equivalent.
\end{proof}

\begin{remark}
The authors are grateful to the referee for pointing out that the conditions (1)--(14) are also equivalent with the statement that $L(E)$ is a polynomial identity (PI) ring. Since a matrix ring over a commutative ring is a PI ring and a finite direct product of PI rings is a PI ring, condition (5) implies that $L(E)$ is a PI ring. Conversely, if $L(E)$ is a PI ring, then it is directly finite so the condition (9) holds.
\end{remark}

Theorem \ref{equivalences1-14} has the following corollary.
\begin{corollary}
If $E$ is a finite no-exit graph, then  $L(E)$ is Baer.
\label{Baer}
\end{corollary}
\begin{proof}
\cite[Theorem 7.55]{Lam} states that for a ring satisfying condition (6) from Theorem \ref{equivalences1-14}, the following conditions are equivalent: being Baer, being right Rickart, and being left Rickart. Since $L(E)$ is Rickart (as a semihereditary ring) and (6) is equivalent to (1), $L(E)$ is Baer.
\end{proof}

Now we turn to the question whether the involution of $L(E)$ extends to $Q(E).$ Recall that a regular ring with involution is said to be \emph{$^\ast$-regular} if the involution is proper (see \cite[Exercise 6A, \S 3]{Be}). Also recall \cite[Theorem 3.1]{Ara_Menal} stating that a $^\ast$-regular ring $R$ must necessarily be finite.

\begin{proposition} If the involution on $K$ is positive definite and if the involution $^\ast$ of $L(E)$ extends from $L(E)$ to $Q(E),$ then this extension is  positive definite and $Q(E)$ is a finite $^\ast$-regular ring.
\label{if_extends_star_regular}
\end{proposition}
\begin{proof}
Let us first show that the involution on $Q(E)$ is positive definite. Assume that for some $n\geq 1$ we have $\sum_{i=1}^n q_{i}^*q_{i}=0$ with $q_{i}\in Q(E)$. Write $p_{i}=q_{i}^*$ so that we have $\sum_{i=1}^n p_{i}p_{i}^*=0$. Assume that there is some $q_{i}$, say $q_{1}$, that is nonzero. In that case $p_{1}$ is nonzero too.

The ring $Q(E)=\Qltot(L(E))$ is a left ring of quotients of $L(E)$ and $L(E)$ is dense in $Q(E)$ as a left $L(E)$-submodule of $Q(E)$ ($L(E)$ embeds into $\Qltot(L(E))$ by \cite[Theorem 4.1]{Stenstrom}). Thus, applying \cite[Exercise 9, p. 284]{Lam} we can find $r\in L(E)$ such that $rp_{1}\neq 0$ and $rp_{i}\in L(E)$ for all $i$.

Now $\sum_{i=1}^n p_{i}p_{i}^*=0$ implies that $0=r(\sum_{i=1}^n p_{i}p_{i}^*)r^*=\sum_{i=1}^n (rp_{i})(rp_{i})^*=0$. But applying \cite[Proposition 2.4]{ARV} we have that $^\ast$ is positive definite in $L(E)$ if the involution on $K$ is positive definite. Therefore we have that $rp_{i}=0$ for all $i$, which is a contradiction with the fact that $rp_{1}\neq 0$. Therefore $q_{i}=0$ for all $i.$ This shows that the involution on $Q(E)$ is positive definite.

For a regular ring to be $^\ast$-regular, it is sufficient for $^\ast$ to be proper. Since $^\ast$ is positive definite on $Q(E),$ $^*$ is proper. Thus, $Q(E)$ is $^\ast$-regular. In addition, $Q(E)$ is finite by \cite[Theorem 3.1]{Ara_Menal}.
\end{proof}

Note that if $E$ is a finite and acyclic graph, $L(E)=Q(E)$ and so the involution extends trivially. This is because $L(E)$ of an acyclic graph is regular by \cite[Theorem 1]{Gene-Ranga}.  A regular ring is equal to its total left (and right) ring of quotients (\cite[Example 1, p. 235]{Stenstrom}). Thus,  $L(E)=Q(E)$ since $Q(E)$ is the total left ring of quotients of $L(E).$

The next example demonstrates that the involution does not always extend from $L(E)$ to $Q(E).$

\begin{example} Let $K$ be any field that has a positive definite involution (for example, complex numbers with conjugated complex involution). Let $E$ be the graph of a single vertex $v$ and two edges $e,f$.
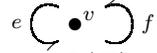
$$\xymatrix{ {\bullet}^{v} \ar@(ul,dl) _{e} \ar@(dr,ur) _{f}   }$$
Then the involution does not extend from $L(E)$ to $Q(E).$
\end{example}
\begin{proof}
If we assume that the involution extends from $L(E)$ to $Q(E),$ $Q(E)$ would be $^\ast$-regular and finite by Proposition \ref{if_extends_star_regular}. But this is not the case since $e^*e=1$ and $ee^*=1-ff^*\neq 1$ because $ff^*\neq0$ (since otherwise $f=f1=ff^*f=0$). So $Q(E)$ is not finite. Hence the involution does not extend from $L(E)$ to $Q(E).$
\end{proof}

Any extension of the involution from $L(E)$ to $Q(E)$ has to be unique as the following proposition shows.
\begin{proposition} If the involution $^\ast$ of $L(E)$ extends from $L(E)$ to $Q(E),$ then this extension is unique.
\end{proposition}
\begin{proof}
Assume that there are two involutions extending the involution of $L(E)$ to $Q(E).$ Then their composition is a ring automorphism of $Q(E)$ that leaves $L(E)$ fixed. Consider the difference $f$ between this automorphism and the identity map. The map $f$ has $L(E)$ in the kernel and so $f$ factors to a map $\overline{f}: Q(E)/L(E)\rightarrow Q(E)$. We claim that $\overline{f}$ is zero since it maps a torsion module into a torsion-free module. Indeed, $Q(E)$ is torsion-free by \cite[Proposition 1.8, p. 198]{Stenstrom} with respect to the torsion theory that makes $Q(E)$ into the total left ring of quotients of $L(E)$ (more details on this torsion theory can be found in \cite{Stenstrom}). The module $Q(E)/L(E)$ is the cokernel of the monomorphism $L(E)\subseteq Q(E),$ the natural map from $L(E)$ to $Q(E)=\Qltot(L(E)).$  The cokernel $Q(E)/L(E)$ is a torsion module by \cite[Lemma 1.5, p. 196]{Stenstrom}.
Thus, $\overline{f}$ is a map from a torsion to a torsion-free module and so it has to be zero. Hence $f$ is zero as well and so the involution extends uniquely.
\end{proof}

The next result shows that further equivalences can be added to the list of those in Theorem \ref{equivalences1-14} in case when the involution on $K$ is positive definite.

Recall that a $^\ast$-ring $R$ is said to be {\em symmetric} if $1+x^*x$ is invertible for every $x.$ In this case, $R$ has a property that for every idempotent $a$ there is a projection (selfadjoint idempotent) $p$ such that $aR=pR$ (\cite[Exercise 7C, p. 9]{Be}).

\begin{theorem} Let $E$ be a finite graph.
If the involution on $K$ is positive definite, then (15)--(19) are equivalent to (1)--(14).
\begin{itemize}
\item[(15)] The involution $^\ast$ extends from $L(E)$ to $Q(E).$
\item[(16)] $Q(E)$ is $^\ast$-regular (for the involution inherited from $L(E))$.
\item[(17)] $Q(E)$ is finite (for the involution inherited from $L(E))$.
\item[(18)] $L(E)$ is extendible (i.e. $^\ast$ extends to $\Qlmax(L(E))$) and $Q(E)=\Qlmax(L(E))$.
\item[(19)] $Q(E)$ is a symmetric $^\ast$-ring (for the involution inherited from $L(E))$.
\end{itemize}
\label{equivalences15-19}
\end{theorem}

\begin{proof}
(5) $\Rightarrow$ (15). Let $\phi$ be the map from $L(E)$ onto $R=\left(\bigoplus_{i=1}^l M_{m_i}(K[x, x^{-1}])\right)\oplus$ $\left(\bigoplus_{j=1}^{k}M_{n_j}(K)\right)$ as in (5). This isomorphism induces the isomorphism (we call it also $\phi$) of the maximal left rings of quotients $Q(E)=\Qlmax(L(E))$ (see Proposition \ref{necessary})
onto $\Qlmax(R)=\left(\bigoplus_{i=1}^l M_{m_i}(K(x))\right)\oplus\left(\bigoplus_{j=1}^{k}M_{n_j}(K)\right).$ For an element $q$ of $Q(E)$ define the involution by $q^*=\phi^{-1}(\phi(q)^*).$ This is well defined by Proposition \ref{necessary}
and defines an involution on $Q(E)$ that extends the one on $L(E)$ also by Proposition \ref{necessary}.

The implication (15) $\Rightarrow$ (16) is Proposition \ref{if_extends_star_regular}.

The implication (16) $\Rightarrow$ (17) is \cite[Theorem 3.1]{Ara_Menal}.

The implication (17) $\Rightarrow$ (1) follows from Proposition \ref{sufficient} since if $Q(E)$ is finite, then $L(E)$ is finite as well.

(15) $\Leftrightarrow$ (18). Note that (1) implies that $Q(E)=\Qlmax(L(E))$ (this fact is in the proof of Proposition \ref{necessary}). Since (15) + (1) implies (18) and (15) implies (1), we have that (15) $\Rightarrow$ (18). Finally, (18) clearly implies (15).

(15) $\Rightarrow$ (19). A $^\ast$-regular ring with 2-proper ($x^*x+y^*y=0$ implies $x=y=0$) involution is symmetric (\cite[Exercise 9A p. 232]{Be}). Thus, a $^\ast$-regular ring with positive definite involution is symmetric. Condition (15) and Proposition \ref{if_extends_star_regular} imply that $Q(E)$ is a $^\ast$-regular ring with positive definite involution.

(19) $\Rightarrow$ (16). A $^\ast$-ring that is symmetric and regular is $^\ast$-regular (since symmetric and Rickart implies Rickart $^*$, and Rickart $^*$ and regular implies $^*$-regular, \cite[Exercises 6A, 7A p. 18]{Be}).
\end{proof}

\begin{remark} Conditions (1) -- (14) imply that $Q(E)$ is Baer (as any left or right self-injective and regular ring is Baer, see \cite[Corollary 7.53]{Lam}). Conditions (15) -- (18) imply that $Q(E)$ is Baer $^\ast$-ring (since $^\ast$-regular and Baer ring is regular Baer $^\ast$-ring).
\end{remark}

\begin{example} If $E$ is a finite no-exit graph, $Q(E)$ does not have to be symmetric if the involution of $K$ is not positive definite. For example, if $E$ is a loop (i.e. graph with one vertex and one edge), then $L_{\Zset_2}(E)=\Zset_2[x,x^{-1}]$ (\cite[Example 1.4]{Gene_Gonzalo1}) and $Q_{\Zset_2}(E)=\Zset_2(x)$. This algebra is not symmetric because $1+xx^*=1+xx^{-1}=1+1=0$ is not invertible.

Also, the involution on  $Q_{\Zset_2}(E)$ is not proper so $Q_{\Zset_2}(E)$ is not $^\ast$-regular. Thus we see that (1)--(14) hold but not (16), for example. So, this example shows that condition (15)--(18) are not necessarily equivalent with (1)--(14) if the involution of the field $K$ is not positive definite.
\end{example}

Theorem \ref{equivalences15-19} extends the class of known rings for which the Handelman's conjecture holds. Recall that the Handelman's conjecture (\cite[Problem 48, p. 380]{G}) is asking if every $^\ast$-regular ring is directly finite and if it is unit-regular. In \cite[Corollary 4.2]{ARV} it was shown that every  $^\ast$-regular Leavitt path algebra is unit-regular. From this result also follows that a $^\ast$-regular Leavitt path algebra is directly finite (since a unit-regular ring is directly finite by \cite[Proposition 5.2]{G}.)

Theorem \ref{equivalences15-19} asserts that the Handelman's conjecture holds for the class of regular algebras of Leavitt path algebras: in case when the involution on $K$ is positive definite and $Q(E)$ is equipped with the involution originating from $L(E),$ the regular algebra $Q(E)$ is $^\ast$-regular and both directly finite and unit-regular.

\section{$Q(E)$ as a ring of quotients}
\label{section_symmetric}

In this section, we explore the properties of $Q(E)$ as both one-sided and symmetric ring of quotients. Let us recall a few facts on rings of quotients first. Let $\ef_r$ be a right Gabriel filter (i.e. a set of right ideals of a ring $R$ that defines a hereditary torsion theory $\tau$, see \cite{Stenstrom} for more details). The right ring of quotients of $R$ with respect to $\ef_r$ (and $\tau$) is denoted by $R_{\ef_r}.$ This ring can be represented as $\dirlim_{I\in\ef_r}\homo(I, \frac{R}{\te(R)})$ where $\te(R)$ is the torsion submodule of $R$. Equivalently, $q$ is an element of $R_{\ef_r}$ if there is a right ideal $I\in \ef_r$ and a right module homomorphism $f:I\rightarrow \frac{R}{\te(R)}$ such that $f(x)=qx$ for every $x\in I.$ A left-sided version is defined similarly.

Now consider an involutive ring $R$. We define a left Gabriel filter $\ef_l$ and a right Gabriel filter $\ef_r$ to be {\em conjugated} if and only if $\ef_l^*=\ef_r$ (i.e. $I\in \ef_r$ if and only if $I^*=\{r^*\ |\ r\in I\}\in \ef_l).$

\begin{proposition}
Let $R$ be an involutive ring and $\ef_l$ and $\ef_r$ conjugated left and right Gabriel filters. If the involution extends to $R_{\ef_r}$ then $R_{\ef_r}$ is also a left ring of quotients and $_{\ef_l}R=R_{\ef_r}.$ Similarly, if it extends to $_{\ef_l}R,$ then $_{\ef_l}R$ is a right ring of quotients and $_{\ef_l}R=R_{\ef_r}.$
\label{right_is_symmetric}
\end{proposition}
\begin{proof}
Let $\tau_r=(\te_r,\f_r)$ and $\tau_l=(\te_l,\f_l)$ denote the torsion theories of right and left $R$-modules that $\ef_r$ and $\ef_l$ determine. Note that the involution maps the left torsion submodule $\te_l(R)$ onto the right torsion submodule $\te_r(R).$ This is because $r\in\te_r(R)$ if and only if $rI=0$ for some right ideal in $\ef_r.$ But then $I^*r^*=0$ and $I^*$ is in $\ef_l$ and so $r^*\in\te_l(R).$ Similarly, $r\in\te_l(R)$ implies that $r^*\in\te_r(R).$

Let now $q$ be in $R_{\ef_r}.$ Since the involution extends to $R_{\ef_r},$ $q^*$ is in $R_{\ef_r}$ as well. So, there is a right ideal $I\in\ef_r$ and a right $R$-homomorphism $f:I\rightarrow R/\te_r(R)$ such that $f(x)=q^*x.$ Then $I^*$ is a left ideal and the map $f^*:I^*\rightarrow R/\te_l(R)$ defined by $f^*(x)=f(x^*)^*$ is such that $f^*(x)=xq.$ Thus $q$ is in $_{\ef_l}R.$ The converse is proven similarly: if $q$ is in $_{\ef_l}R$ and $I$ a left ideal with a homomorphism $f:I\rightarrow R/\te_l(R)$ such that $f(x)=xq,$ then $I^*$ is a right ideal and $f^*:I^*\rightarrow R/\te_r(R)$ given by $f^*(x)=f(x^*)^*$ is such that $f^*(x)=q^*x.$ Thus $q^*$ is in $R_{\ef_r}.$ But then $q=(q^*)^*$ is in $R_{\ef_r}$ as well.
\end{proof}

If $\ef_l$ and $\ef_r$ are left and right Gabriel filters, the symmetric filter $_l\ef_r$ induced by $\ef_l$ and $\ef_r$ is defined to be the set of (two-sided) ideals of $R$ containing ideals of the form $IR+RJ$, where $I\in\ef_l$ and $J\in \ef_r$ (equivalently, the set of right ideals of $R\otimes_{\Zset}R^{op}$ containing ideals of the form $J\otimes R^{op}+R\otimes I$).
This defines a Gabriel filter by \cite[p. 100]{Ortega_thesis}. The corresponding torsion theory induced by $\tau_l$ and $\tau_r$ is denoted by  $_l\tau_r.$ If $M$ is an $R$-bimodule, $\te_l(M),$ $\te_r(M)$ and $_l\te_r(M)$ torsion submodules of $M$ for $\tau_l,$ $\tau_r$ and  $_l\tau_r$ respectively, then $_l\te_r(M)=\te_l(M)\cap\te_r(M).$

In \cite{Ortega_thesis} and \cite{Ortega_paper}, Ortega considers the symmetric ring of quotients $_{\ef_l}R_{\ef_r}$ (or $_lR_r$ for short) with respect to $_l\ef_r$ to be \[_lR_r=\dirlim_{K\in _l\ef_r}\;  \homo(K, \frac{R}{_l\te_r(R)})\]
where the homomorphisms in the formula are $R$-bimodule homomorphisms. Ortega shows that an equivalent approach can be obtained considering compatible pairs of homomorphisms. If $I\in\ef_l,$ $J\in\ef_r,$ $f:I\rightarrow\, \overline{R}$ and $g:J\rightarrow \overline{R}$ are homomorphisms with $\overline{R}=R/_l\te_r(R),$ then $(f,g)$ is a {\em compatible pair} if $f(i)j=ig(j)$ for all $i\in I,$ $j\in J.$ Define the equivalence relation by $(f,g)\sim(h,k)$ if and only if $f|_{I'}=h|_{I'}$ for some $I'\in \ef_l$ and $g|_{J'}=k|_{J'}$ for some $J'\in \ef_r.$ Then there is a bijective correspondence between elements of $_lR_r$ and the equivalence classes of compatible pairs of homomorphisms (\cite[Proposition 4.37]{Ortega_thesis} or \cite[Proposition 1.4]{Ortega_paper}) such that an element $q$ of $_lR_r$ can be represented by $(f,g)$ $f:I\rightarrow\overline{R}$ and $g:J\rightarrow\overline{R}$ if $f(i)=iq$ and $g(j)=qj$ for all $i\in I$ and $j\in J.$ In \cite{Ortega_thesis} and \cite{Ortega_paper}, Ortega further defines the symmetric module of quotients $_{\ef_l}M_{\ef_r}$ with respect to left and right Gabriel filters $\ef_l$ and $\ef_r$ of an $R$-bimodule $M$.

We prove the following property of symmetric rings of quotients of involutive rings.
\begin{proposition} Let $R$ be a ring with involution and $\ef_l$ and $\ef_r$ left and right Gabriel filters conjugated to each other.
\begin{enumerate}
\item[(i)] The involution extends to the symmetric ring of quotients $_lR_r.$

\item[(ii)] If $R$ is $\tau_r$-torsion-free ($\tau_l$-torsion-free) and the involution extends to $R_{\ef_r}$ ($_{\ef_l}R$)
then $R_{\ef_r}=\,_{\ef_l}R=\, _lR_r.$
\end{enumerate}
\label{involution_extends_to_symmetric}
\end{proposition}
\begin{proof}
First note that $_l\te_r(R)^*=(\te_l(R)\cap\te_r(R))^*=\te_r(R)\cap\te_l(R)=\,_l\te_r(R)$ if $\ef_l$ and $\ef_r$ are conjugated to each other.

(i) Let $(f,g)$ be a compatible pair representing the equivalence class of an element $q$ of the symmetric ring of quotients $_lR_r.$ Let $\overline{R}=R/_l\te_r(R),$ $I\in\ef_l,$ $J\in\ef_r,$ and $f:I\rightarrow\, \overline{R}$ and $g:J\rightarrow \overline{R}$ be homomorphisms with $f(i)=iq$ and $g(j)=qj.$  We can assume that $J=I^*$ (otherwise, we can replace the pair $(f,g)$ by an equivalent pair of homomorphisms defined on a left ideal generated by $I\cup J^*$ and a right ideal generated by $I^*\cup J$).

To define $q^*,$ consider a pair $(g^*, f^*)$ with $g^*:I\rightarrow \overline{R}$ and $f^*:I^*\rightarrow \overline{R}$ defined by $g^*(i)=g(i^*)^*$ and $f^*(j)=f(j^*)^*.$ It is a compatible pair since $g^*(i)j=g(i^*)^*j=(j^*g(i^*))^*=(f(j^*)i^*)^*=if(j^*)^*=if^*(j).$ This compatible pair represents the element $q^*$ since
$g^*(i)=g(i^*)^*=(qi^*)^*=iq^*$ and $f^*(j)=f(j^*)^*=(j^*q)^*=q^*j.$

(ii) Note that $ _lR_r$ embeds in $R_{\ef_r}$ by sending the class of a compatible pair $(f,g)$ to the equivalence class of $g.$ If the involution extends to $R_{\ef_r},$ then $R_{\ef_r}=\,_{\ef_l}R$ by Proposition \ref{right_is_symmetric}.

If $R$ is $\tau_r$-torsion-free, $\te_r(R)=0$ and so $\te_l(R)=(\te_r(R))^*=0$ and $_l\te_r(R)=\te_l(R)\cap\te_r(R)=0$ also.
We claim that $R_{\ef_r}$ embeds into $ _lR_r.$ To see this, let $q\in R_{\ef_r}$ be represented by $g:J\rightarrow R$ for some $J\in\ef_r.$ Since $q$ is also in $_{\ef_l}R,$ there is a left ideal $I$ and a map $f:I\rightarrow R$ that represents $q$ as an element of $_{\ef_l}R.$ Then $(f, g)$ is a compatible pair representing $q$ as an element of $_lR_r$ since
$f(i)j=(iq)j=i(qj)=ig(j).$
\end{proof}

If $\ef_l$ is the filter of dense left and $\ef_r$ a filter of dense right ideals, the symmetric ring of quotients induced by $\ef_l$ and $\ef_r$ is the maximal symmetric ring of quotients  $\Qmax(R)$ (introduced in \cite{Utumi}, studied in \cite{Lanning} and \cite{Ortega_paper_Qsimmax}). In particular, the classes of dense left and right ideals are conjugated so Proposition \ref{right_is_symmetric} and Proposition \ref{involution_extends_to_symmetric} can be applied to these filters.

The following lemma shows that perfect symmetric rings of quotients of involutive rings are also obtained via conjugated filters.
\begin{lemma} If $R$ is an involutive ring and $S$ a perfect symmetric ring of quotients with an injective localization map $q:R\rightarrow S$, then the involution extends to $S$ making $q$ a $^*$-homomorphism. With this involution, the left filter  $\ef_l=\{I | Sq(I)=S\}$ is conjugated to the right filter $\ef_r=\{J | q(J)S=S\}.$
\label{symmetric_perfect_have_involution}
\end{lemma}
\begin{proof} If  $S$ a perfect symmetric ring of quotients with localization map $q$, then $S$ is a symmetric ring of quotients with respect to the torsion theory induced by the left filter $\ef_l=\{I | Sq(I)=S\}$ and the right filter  $\ef_r=\{J | q(J)S=S\}$
by \cite[Theorem 4.1]{Lia_symmetric}. Note that if $I\in \ef_l,$ and $J\in \ef_r,$ then the left ideal $I+J^*$ is in $\ef_l$ and the right ideal $I^*+J$ is in $\ef_r.$ So, a compatible pair $(f,g)$ such that $f:I\rightarrow R$ and $g: J\rightarrow R$ can be exchanged by a compatible pair equivalent to it (call it $(f,g)$ again) defined on two ideals conjugated to each other. For $s\in S$ represented by such $(f,g),$ we can define $s^*\in S$ by representing it with the compatible pair $(g^*, f^*)$ defined in the same way as in the proof of Proposition \ref{involution_extends_to_symmetric}.

This definition makes $q$ a $^*$-homomorphism: for $r\in R$ the image $q(r)$ can be represented by a compatible pair $(R_r, L_r)$ where $L_r$ is the left and $R_r$ is the right multiplication by $r.$ Then the element $q(r)^*$ is represented by $(L_r^*, R_r^*)$ which is exactly $(R_{r^*}, L_{r^*})$ that represents $q(r^*).$ Thus $q(r^*)=q(r)^*$ for all $r\in R.$

With such involution on $S$, it is easy to see that $I\in \ef_l$ if and only if $I^*\in \ef_r.$ So, the filters are conjugated.
\end{proof}

In \cite{Lia_symmetric}, the total symmetric ring of quotients $\Qtot$ is defined as a symmetric version of the total one-sided rings of quotients $\Qrtot$ and $\Qltot.$ From Lemma \ref{symmetric_perfect_have_involution}, it follows that the involution extends to the total symmetric ring of quotients $\Qtot(R).$ Note that $R$ always embeds into $\Qtot(R)$ (\cite[Theorem 5.1]{Lia_symmetric}) so the assumptions of Lemma \ref{symmetric_perfect_have_involution} are satisfied. Also $R$ also embeds in $\Qrtot(R),$ $\Qltot(R)$ (see \cite[Theorem 4.1]{Stenstrom}) so Proposition \ref{right_is_symmetric} and Proposition \ref{involution_extends_to_symmetric} can be applied as well. Thus, the following corollary holds.

\begin{corollary} Let $R$ be an involutive ring.
\begin{enumerate}
\item[(i)] The involution extends to $\Qmax(R)$ and $\Qtot(R).$

\item[(ii)] If the involution extends to $\Qrmax(R)$ or $\Qlmax(R),$ then $\Qrmax(R)=\Qlmax(R)=\Qmax(R).$

\item[(iii)] If the involution extends to $\Qrtot(R)$ or $\Qltot(R),$ then $\Qrtot(R)=\Qltot(R)=\Qtot(R).$
\end{enumerate}
\label{quotients_of_involutive_ring}
\end{corollary}

Let us consider Leavitt path algebras now. We have seen that the equivalent conditions (1)--(14) imply that $Q(E)=\Qlmax(L(E))=\Qltot(L(E)).$ In case that (15)--(19) hold as well, the involution extends to $Q(E).$ Thus the following holds.

\begin{corollary}
If the involution on $K$ is positive definite, then conditions (1)--(19) are equivalent with
\begin{itemize}
\item[(20)] $Q(E)$ is a symmetric ring of quotients induced by conjugated filters.

\item[(21)] $Q(E)=\Qmax(L(E)).$

\item[(22)] $Q(E)=\Qtot(L(E)).$
\end{itemize}
\label{Q(E)_as_quotient}
\end{corollary}
\begin{proof}
Any of (20), (21) or (22) implies (15) by Proposition \ref{involution_extends_to_symmetric}.
Conversely, (15) implies (21) and (22) by Corollary \ref{quotients_of_involutive_ring}.  Finally, (21) or (22) trivially imply (20).
\end{proof}

If $E$ is a finite graph, $Q(E)$ is the total left ring of quotients of $L(E)$ even if it may not be equal to $\Qlmax(L(E))$ in case that $L(E)$ is not noetherian. However, since $L(E)$ is hereditary (thus semihereditary as well), the description of the total left ring of quotients from \cite[Theorem 12]{Lia_quotients} can be used to describe $Q(E)$ via $\Qlmax(L(E)).$ Thus,
\[Q(E)=\{q\in \Qlmax(L(E))\ |\ Iq\subseteq R\mbox{ for a left ideal }I\mbox{ with }\Qlmax(L(E))I=\Qlmax(L(E))\}.\]

\section{Further equivalences and $K_0$-theorem}
\label{section_K_0}

Rings with involution for which the involution extends to one-sided maximal rings of quotients have been studied in the past (e.g. \cite{Handelman}, \cite{Pyle}). For Leavitt path algebras with (1)--(22), the involution can be extended to the left maximal rings of quotients. This fact implies some further properties that we discuss in this section. In the following proposition and its proof, $\Qrmax$ stands for $\Qrmax(L(E)),$ $\Qlmax$ stands for $\Qlmax(L(E))$ and similar abbreviations are used for the total and classical rings of quotients of $L(E)$ as well.

\begin{proposition} If $E$ is a finite graph and $K$ a field with positive definite involution, then the following conditions are equivalent with (1)--(22).
\begin{itemize}
\item[(23)] $\Qrmax=\Qltot.$

\item[(24)] Every finitely generated nonsingular $L(E)$-module is projective.

\item[(25)] $\Qrmax\otimes_{L(E)}\Qrmax\cong\Qrmax$ via $q\otimes s\mapsto qs$ and $\Qrmax$ is flat as a right and left $R$-module.

\item[(26)] $\Qrmax=\Qltot=\Qrtot.$

\item[(27)] $M_n(L(E))$ is right strongly Baer (i.e. every complemented right ideal is generated by an idempotent) for every $n.$

\item[(28)] Every finitely generated $L(E)$-module $M$ is such that $M\otimes_{L(E)} Q(E)= M\otimes_{L(E)} \Qrmax.$

\item[(29)] $\Qrmax=\Qlcl.$
\end{itemize}
Moreover, these conditions imply that $M\otimes_{L(E)}\Qrmax$ is isomorphic to the injective envelope $E(M)$ for any nonsingular right $L(E)$-module $M.$ Also, the following rings of quotients are all equal to $Q(E)$
$\Qrmax=\Qlmax=\Qmax=\Qrtot=\Qltot=\Qtot=\Qrcl=\Qlcl.$
\label{extendible_LPAs}
\end{proposition}

\begin{proof}
First we prove (23) $\Rightarrow$ (24) $\Rightarrow$ (25) $\Rightarrow$ (26) $\Rightarrow$ (23). Then we show that these conditions are equivalent with (1)--(22). Finally, we demonstrate that (27)--(29) are equivalent to the rest.

Condition (23) implies that $\Qrmax$ is a perfect left ring of quotients of $L(E).$ Since $L(E)$ is (right) semihereditary, condition (24) follows by \cite[Corollary 7.4, p.259]{Stenstrom}.

Also by \cite[Corollary 7.4, p.259]{Stenstrom} (see the last sentence of it), (24) implies that $\Qrmax$ both a perfect left and a perfect right ring of quotients. Since such quotients have property (25) (see \cite[Theorem 4.1]{Lia_symmetric}), (24) implies (25).

Condition (25) implies that $\Qrmax$ is a perfect left and a perfect right ring of quotients. In that case, it is also a perfect symmetric ring of quotients. By Lemma \ref{symmetric_perfect_have_involution}, the involution extends to $\Qrmax.$ Hence $\Qrmax=\Qlmax=\Qmax$ by Corollary \ref{quotients_of_involutive_ring}. Since (25) also implies that $\Qrmax\subseteq\Qrtot$ and $\Qrmax\subseteq\Qltot,$ the equalities follow since $\Qrmax\subseteq\Qrtot\subseteq\Qrmax$ and $\Qrmax\subseteq\Qltot\subseteq\Qlmax=\Qrmax.$ So, (26) holds.

Condition (26) trivially implies (23).

Condition (26) implies (22). This is because $\Qltot=\Qrtot$ implies that $\Qltot=\Qrtot=\Qtot$ and this implies (22) because $Q(E)$ is always equal to $\Qltot.$ Conversely, if (22) holds, then (21) and (18) hold as well and so $\Qtot$ is equal to $\Qltot=\Qlmax,$ thus $\Qrtot$ as well, and $\Qtot$ is also equal to $\Qmax$ and $\Qrmax$ by part (ii) of Corollary \ref{quotients_of_involutive_ring}. Thus (26) holds.

\cite[Theorem 2.4]{Evans_all} states that (27) is equivalent to the statement that $L(E)$ is right semihereditary with $\Qrmax$ that is a perfect left and a perfect right ring of quotients. Since this last condition is precisely (25) and $L(E)$ is indeed right semihereditary, (25) and (27) are equivalent.

(23) $\Leftrightarrow$ (28). From $Q(E)=\Qltot,$ we have that (23) implies (28). Taking $M=L(E)$ in (28), we obtain (23).

Condition (13) implies that $L(E)$ is a semiprime left and right Goldie ring. This is because $L(E)$ is semiprime, left and right nonsingular and has finite uniform dimension if we assume (13), so it is a semiprime two-sided Goldie ring (\cite[Theorem 11.13]{Lam}). Thus, it is left and right Ore with $\Qlmax=\Qlcl$ and $\Qrmax=\Qrcl$ (\cite[Corollary 13.15]{Lam}). A left and right Ore ring has $\Qrcl=\Qlcl$  (\cite[Remark 10.17]{Lam}). Thus $\Qrmax=\Qlcl$ and (29) follows. Conversely, we show that (29) implies (24). Since $\Qlcl$ is a perfect left ring of quotients (\cite[Example p.230]{Stenstrom}), (29) implies that $\Qrmax$ is a perfect left ring of quotients. Since $L(E)$ is right semihereditary, (24) follows by \cite[Corollary 7.4]{Stenstrom}.

The last sentence of the proposition holds by \cite[Corollary 2.8, p. 248]{Stenstrom}.
\end{proof}

\begin{remark}
Note that the algebra $M_n(L(E))$ is isomorphic to the Leavitt path algebra
$L(M_n E)$ where $M_nE$ is the graph obtained by adding the oriented line of
length $n-1$ to every vertex of $E$ (see \cite[Definition 9.1 and Proposition
9.3]{Abrams_Tomforde}). Note that this gives us $M_n(Q(E))\cong Q(M_nE)$ as
well. This follows from the fact that $Q(E)=\Qltot(L(E))$ and that
$M_n(\Qltot(R))\cong \Qltot(M_n(R))$  for any ring $R$.
\end{remark}

We use Proposition \ref{extendible_LPAs} to obtain a specific description of the
inverse of the isomorphism $V(L(E))\rightarrow V(Q(E))$ of the monoids of
equivalence classes of finitely generated projectives over $L(E)$ and $Q(E).$ In
\cite[Theorem 3.5]{AMP}, it is shown that there is natural isomorphism between
$V(L(E))$ and an abelian monoid $M_E$ defined via the generators $\{a_v\ |\ v\in
E^0\}$ and subject to relations
\begin{center}
$a_v=\sum a_{r(e)}$ for all $e\in E^1$ with $v=s(e)$ and every $v\in E^0$ that emits edges
\end{center} if $E$ is row-finite. \cite[Theorem 3.1]{AB} proves that there is a canonical isomorphism between $M_E$ and $V(Q(E)).$ In particular, from the proof it follows that the isomorphism $\varphi: V(L(E))\rightarrow V(Q(E))$ is induced by the map  $P\mapsto P\otimes_{L(E)}Q(E).$ Our goal is to prove that the inverse of the isomorphism $\varphi$ is induced by $P\mapsto P\cap L(E)^n$ if $P$ is a finitely generated projective $Q(E)$-module that can be embedded in $Q(E)^n.$

The relation between $L(E)$ and $Q(E)$ in certain ways parallels the one between a finite von Neumann algebra (or a Baer $^\ast$-ring satisfying axioms as in \cite{Lia_dimension}) and its algebra of affiliated operators (or the regular ring of a Baer $^\ast$-ring). So, the proof of Lemma \ref{complements=summands} in what follows parallels that of \cite[Lemma 8]{Lia_dimension}. Also, the proof of Theorem \ref{K0_theorem} parallels the proof of \cite[Corollary 25]{Lia_dimension}. First, let us recall a few preliminary facts. Let $R$ be any ring and $M$ a nonsingular $R$-module. By \cite[Corollary 7.44']{Lam}, there is a one-to-one correspondence between closed submodules of $M$ and those of $E(M)$ (for the definition of closed submodules, see \cite[Definition 7.31]{Lam}). By the remark after the proof of \cite[Corollary 7.44']{Lam}, the closed submodules of $E(M)$ are precisely the direct summands of $E(M).$ Moreover, if $N$ is a closed module of $M$, then its closure in $E(M)$ is a copy of the injective envelope $E(N).$ By \cite[Proposition 7.44]{Lam}, a submodule $N$ of $M$ is closed if and only if $N$ is a complement in $M$. This gives us a one-to-one correspondence
 \[\{\mbox{complements in }M\} \longleftrightarrow\{\mbox{direct summands of }E(M)\}\] given by
$N \mapsto$ the closure of $N$ in $E(M)$ that is equal to a copy of $E(N).$ The inverse map is given by $L\mapsto L\cap M$ (see the proof of \cite[Corollary 7.44']{Lam}). This gives us the following lemma.

\begin{lemma} Let $K$ be a field with positive definite involution, $L(E)$ be a noetherian Leavitt path algebra and let $P$ be a right $L(E)$-submodule of $L(E)^n$ for some non-negative $n$.
The following conditions are equivalent:
\begin{enumerate}
\item[(i)] $P$ is a complement in $L(E)^n.$

\item[(ii)] $P$ is closed in $L(E)^n.$

\item[(iii)] There is an idempotent $p\in M_n(L(E))$ such that $P=$ im $p.$

\item[(iv)] $P$ is a direct summand of $L(E)^n.$
\end{enumerate}
\label{complements=summands}
\end{lemma}
\begin{proof} Since $L(E)$ is a nonsingular ring (\cite[Proposition 4.1]{Mercedes_quotients}), $L(E)^n$ is a nonsingular module. Thus, (i) and (ii) are equivalent as noted in the paragraph preceding the lemma. It is also clear that (iii) and (iv) are equivalent.

Since a direct summand is a complement, (iv) implies (i). Finally, we show that (ii) implies (iv). Let $P$ be a closed submodule of  $L(E)^n$. Then $L(E)^n/P$ is nonsingular by \cite[Theorem 7.28]{Lam}. Thus, $L(E)^n/P$ is finitely generated nonsingular module. Then  $L(E)^n/P$ is projective by condition (24) of Proposition \ref{extendible_LPAs}. So, the embedding of $P$ into $L(E)^n$ splits and thus $P$ is a direct summand of $L(E)^n$.
\end{proof}

Now we can prove the following theorem.

\begin{theorem} Let $K$ be a field with positive definite involution and $L(E)$ a noetherian Leavitt path algebra.
\begin{enumerate}
\item[(i)] For every finitely generated nonsingular (equivalently projective) $L(E)$-module $P$, there is
a one-to-one correspondence \[\{\mbox{direct summands of }P\}
\longleftrightarrow\{\mbox{direct summands of }E(P)=P\otimes_{L(E)}Q(E)\}\] given by
$N \mapsto$ $N\otimes_{L(E)}Q(E) = E(N).$ The inverse map is given by
$L\mapsto L\cap P.$

\item[(ii)] The isomorphism $\varphi: V(L(E))\rightarrow V(Q(E))$ induced by the map  $P\mapsto P\otimes_{L(E)}Q(E)$ has the inverse induced by $P\mapsto P\cap L(E)^n$ if $P$ is a finitely generated projective $Q(E)$-module that can be embedded in $Q(E)^n.$
\end{enumerate}
\label{K0_theorem}
\end{theorem}
\begin{proof}
(i) follows directly from Lemma \ref{complements=summands}, the paragraph preceding it and the fact that $P\otimes_{L(E)}Q(E)$ is the injective envelope of $P$ for a nonsingular $P$ (see the last sentence of Proposition \ref{extendible_LPAs}). (ii) follows directly from (i).
\end{proof}

\section{Isomorphism conjecture}
\label{section_isomorphism_conjecture}

A well-known theorem by Gardner \cite[Theorem B]{Gardner} shows that if two $C^*$-algebras are isomorphic as algebras, then they are isomorphic as $^\ast$-algebras. With this in mind, in \cite[p. 20]{Abrams_Tomforde}, G. Abrams and M. Tomforde posed several isomorphism conjectures.

\begin{C1}{\rm If  $L_{\Cset}(E)\cong L_{\Cset}(F)$ as algebras, then $L_{\Cset}(E)\cong L_{\Cset}(F)$
as $^\ast$-algebras.}
\end{C1}

\begin{C2}{\rm If $L_{\Cset}(E)\cong L_{\Cset}(F)$ as algebras, then $C^*(E)\cong C^*(F)$ as $^\ast$-algebras.}
\end{C2}

\begin{IC}{\rm If
$L_{\Cset}(E)\cong L_{\Cset}(F)$ as rings, then $C^*(E)\cong C^*(F)$
as $^\ast$-algebras.}
\end{IC}

It is known (see \cite[Corollary 4.5]{Abrams_Tomforde}) that if $L_{\Cset}(E)\cong L_{\Cset}(F)$ as $^\ast$-algebras, then $C^*(E)\cong C^*(F)$ as $^\ast$-algebras. Thus, (C1) implies (C2).

For these conjectures, it is assumed that the involution on $\Cset$ is fixed to be the complex-conjugate involution \cite[Definition 2.1]{Abrams_Tomforde}. If the involution in (C1) is not fixed, the conjecture trivially fails as the next example shows.

\begin{example}
Let $id$ denote the identity involution on $\Cset$ and let $\hskip.1cm^{\overline{\hskip0.2cm}}$ denote the complex-conjugate involution. Since $(\Cset, id)\cong(\Cset,  \hskip.1cm^{\overline{\hskip0.2cm}})$ as fields, $L_{(\Cset, id)}(E)\cong L_{(\Cset,  \hskip.1cm^{\overline{\hskip0.2cm}})}(E)$ as algebras. However, $L_{(\Cset, id)}(E)$ and $L_{(\Cset,  \hskip.1cm^{\overline{\hskip0.2cm}})}(E)$ are not isomorphic as $^\ast$-algebras: if $f$ is a $^\ast$-isomorphism then $i=if(1)=f(i)=f(i^*)=f(i)^*=(f(1)i)^* = i^*f(1)^*=-if(1)=-i$. A contradiction.
\end{example}

Working with general Leavitt path algebras does not imply any specific ties to the field of complex numbers equipped with the complex-conjugated involution. Thus, we generalize (C1) to a conjecture we call the Generalized Isomorphism Conjecture. We also add a stronger version of it -- the Generalized Strong Isomorphism Conjecture.

\begin{GIC}{\rm Let $K$ be a field with (fixed) involution. If  $L_{K}(E)\cong L_{K}(F)$ as algebras, then $L_{K}(E)\cong L_{K}(F)$ as $^\ast$-algebras.}
\end{GIC}

\begin{GSIC}{\rm Let $K$ be a field with (fixed) involution. If  $L_{K}(E)\cong L_{K}(F)$ as rings, then $L_{K}(E)\cong L_{K}(F)$ as $^\ast$-algebras.}
\end{GSIC}

Note that (GIC) implies (C1) and (C2) and (GSIC) implies (IC) (by \cite[Corollary 4.5]{Abrams_Tomforde}) and (GIC).

As we have seen in the example above, (GIC) and (GSIC) fail if the involution on $K$ is not fixed. Also, it is easy to see that a fixed isomorphism of Leavitt path algebras does not have to be a $^\ast$-isomorphism.

\begin{example}
Let $K$ be any field of characteristic different from 2 with any involution. Consider the Leavitt path algebra $L_K(E)$ over the 2-line graph.
$$\xymatrix{ ^{v}{\bullet} \ar[r]^e  &{\bullet}^{w} }$$
$L_K(E)$ is $^\ast$-isomorphic to the algebra $M_2(K)$ of $2\times 2$ matrices over $K$ via the isomorphism $\phi$ of Proposition \ref{necessary}. Now consider any invertible but non-unitary matrix from $M_2(K).$ For example, we can take $A=\left[\begin{array}{cc}2 & 0\\ 0 & 1 \end{array}\right]$ since char$(K)\neq 2.$ The map $f_A(X)=AXA^{-1}$ is an automorphism of $M_2(K)$ that is not a $^\ast$-isomorphism. Then $\phi^{-1}f_A\phi$ is an automorphism of $L_K(E)$ that is not a $^\ast$-isomorphism.
\end{example}

In \cite{Abrams_Tomforde}, it is shown that (C1) and (IC) hold when $E$ and $F$ are acyclic graphs \cite[Proposition 7.4]{Abrams_Tomforde}. It is also shown that (IC) holds if $E$ and $F$ are row-finite cofinal graphs with at least one cycle and such that every cycle has an exit \cite[Proposition 8.5]{Abrams_Tomforde}.

We give a positive answer for (GSIC) (thus for (GIC) and (IC) as well) for the Leavitt path algebras of finite no-exit graphs.

\begin{theorem}\label{isomorphismconjecture}
Let $E$ and $F$ be finite no-exit graphs. The following conditions are equivalent.
\begin{enumerate}
\item[{\rm (i)}] $L(E)\cong L(F)$ (as rings).
\item[{\rm (ii)}] $L(E)\cong L(F)$ (as algebras).
\item[{\rm (iii)}] $L(E)\cong L(F)$ (as $^\ast$-algebras).
\end{enumerate}
\end{theorem}
\begin{proof}
Suppose that $L(E)\cong L(F)$ as rings. By Theorem \ref{noetherian_NE_theorem}, there exist integers $l$, $k$, $l'$, $k'$, $m_i$, $n_j$, $m'_i$, $n'_j$ such that $L(E)$ is isomorphic to $R=\left(\bigoplus_{i=1}^l M_{m_i}(K[x, x^{-1}])\right)\oplus
\left(\bigoplus_{j=1}^k M_{n_j}(K)\right)$ and $L(F)$ isomorphic to $S=\left(\bigoplus_{i=1}^{l'} M_{m'_i}(K[x, x^{-1}])\right)\oplus
\left(\bigoplus_{j=1}^{k'} M_{n'_j}(K)\right)$ as algebras. By Proposition \ref{necessary}, these isomorphisms are $^\ast$-isomorphisms. Denote the $^*$-isomorphism $L(E)\to R$ by $\phi_E$ and the $^*$-isomorphism $L(F)\to S$ by $\phi_F.$ The ring isomorphism $L(E)\cong L(F)$ induces the ring isomorphism $R\cong S.$ Denote this last isomorphism by $\Phi.$

We show that $l=l'$, $k=k'$ and the sizes of the matrices match after reordering. Thus, there is a $^\ast$-isomorphism $f$ between $R$ and $S$ and the $^\ast$-isomorphism $\phi_F^{-1}f\phi_E$ from $L(E)$ onto $L(F)$.

First assume that $l>0$ and consider the ideal $I:=M_{m_1}(K[x,x^{-1}])\oplus 0 \oplus \dots \oplus 0.$
There exist ideals $I_i$ of $K[x,x^{-1}]$ and $J_j$ of $K$ such that $\Phi(I)=\left(\bigoplus_{i=1}^{l'} M_{m'_i}(I_i)\right)\oplus\left(\bigoplus_{j=1}^{k'} M_{n'_j}(J_j)\right)$.

In particular we have the ring isomorphism $$M_{m_1}(K[x,x^{-1}])\cong \left(\bigoplus_{i=1}^{l'} M_{m'_i}(I_i)\right)\oplus \left(\bigoplus_{j=1}^{k'} M_{n'_j}(J_j)\right),$$ which yields, by taking centers, the ring isomorphism $$K[x,x^{-1}]{\rm Id}_{m_1}\cong \left(\bigoplus_{i=1}^{l'} Z(I_i) {\rm Id}_{m'_i}\right)\oplus
\left(\bigoplus_{j=1}^{k'} Z(J_j) {\rm Id}_{n'_j}\right),$$ where for any $m\in {\mathbb N}$, ${\rm Id}_m$ denotes the identity
matrix of size $m$. But since $Z(I_i)=I_i,$ $Z(J_j)=J_j$, and $K[x,x^{-1}]$ does not have zero divisors, the equation
above is only possible if all but one $I_i$ or $J_j$ are zero. Assume $I_i\neq 0$ for some $i$. Then $${\mathbb M}_{m_1}(K[x,x^{-1}])\cong I\cong
\Phi(I)=0\oplus \dots\oplus 0 \oplus {\mathbb M}_{m'_i}(I_i) \oplus 0 \dots \oplus 0\cong {\mathbb M}_{m'_i}(I_i).$$

Apply \cite[Exercise 14, p. 480]{Lam} to get that $m_1=m'_i$ and that $K[x,x^{-1}]\cong I_i$. In particular, $I_i$
is a unital ring. Since $I_i$ consists of polynomials, a degree argument shows that the only identity element possible for $I_i$ is $1\in K[x,x^{-1}]$. So $I_i$ must equal $K[x,x^{-1}]$.

Similarly we can see that if $J_j\neq 0$, then we would have that $K[x,x^{-1}]\cong J_j.$ This is impossible since $J_{j}$ is either $0$ or $K$, so this case cannot happen.

Then, we have the following equality: $$\Phi(I)=0\oplus \dots\oplus 0 \oplus
{\mathbb M}_{m_1}(K[x,x^{-1}]) \oplus 0 \dots \oplus 0.$$ So we can mod out both $I$ and $\Phi(I)$ in the isomorphism $\Phi$ to obtain an induced isomorphism in the quotients, therefore completely removing one Laurent polynomial-type matrix component on each side.

This, in particular, shows that $l'>0.$ In the same way, by using $\Phi^{-1},$ we obtain that $l'>0$ implies $l>0$ as well. Therefore we have $l>0$ if and only if $l'>0$. By following a descending process as above, this shows that $l=l'$ and that there exists a permutation $\sigma$ such that $\sigma(m_{i})=m'_i$.

Furthermore, after conveniently removing all the matrices over Laurent polynomials we are left with a ring isomorphism

$$\overline{\Phi}:\left(\bigoplus_{j=1}^k M_{n_j}(K)\right) \to \left(\bigoplus_{j=1}^{k'} M_{n'_j}(K)\right).$$

In this situation we can apply Wedderburn-Artin Theorem to readily have that $k=k'$ and that there exists a permutation $\tau$ such that $\tau(n_j)=n'_j$.

All this, together with Proposition \ref{necessary}, shows that (i) $\Rightarrow$ (iii). The implications (iii) $\Rightarrow$ (ii) $\Rightarrow$ (i) are trivial.
\end{proof}

We provide here an affirmative answer to both (GSIC) and (IC) for the class of Leavitt path algebras considered in this paper.

\begin{corollary}\label{gsic_holds} (GSIC) holds for the class of noetherian Leavitt path algebras. In particular, (IC) holds for the class of finite no-exit graphs.
\end{corollary}
\begin{proof}
Apply Theorem \ref{isomorphismconjecture} and \cite[Corollary 4.5]{Abrams_Tomforde}.
\end{proof}

Finally, it is interesting to note the following. On one hand, the Leavitt path algebras over finite no-exit graphs satisfy (IC) by Corollary \ref{gsic_holds}. On the other hand, Leavitt path algebras over row-finite and cofinal graphs in which every cycle has an exit also satisfy (IC) by \cite[Proposition 8.5]{Abrams_Tomforde} (the assumption that there has to be at least one cycle in \cite[Proposition 8.5]{Abrams_Tomforde} can be dropped since acyclic graphs also satisfy (IC)). This gives us that the Leavitt path algebras over the finite and cofinal graphs on the two opposite sides of the spectrum (either no exits at all or exits from every cycle) both satisfy (IC). This gives reasons for hope that (IC) may hold for Leavitt path algebras of graphs in between these two extreme cases.

\end{document}